\documentclass[reqno,12pt]{amsart}
\def\vint_#1{\mathchoice%
          {\mathop{\kern 0.2em\vrule width 0.6em height 0.69678ex depth -0.58065ex
                  \kern -0.8em \intop}\nolimits_{\kern -0.4em#1}}%
          {\mathop{\kern 0.1em\vrule width 0.5em height 0.69678ex depth -0.60387ex
                  \kern -0.6em \intop}\nolimits_{#1}}%
          {\mathop{\kern 0.1em\vrule width 0.5em height 0.69678ex
              depth -0.60387ex
                  \kern -0.6em \intop}\nolimits_{#1}}%
          {\mathop{\kern 0.1em\vrule width 0.5em height 0.69678ex depth -0.60387ex
                  \kern -0.6em \intop}\nolimits_{#1}}}
\def\vintslides_#1{\mathchoice%
          {\mathop{\kern 0.1em\vrule width 0.5em height 0.697ex depth -0.581ex
                  \kern -0.6em \intop}\nolimits_{\kern -0.4em#1}}%
          {\mathop{\kern 0.1em\vrule width 0.3em height 0.697ex depth -0.604ex
                  \kern -0.4em \intop}\nolimits_{#1}}%
          {\mathop{\kern 0.1em\vrule width 0.3em height 0.697ex de pth -0.604ex
                  \kern -0.4em \intop}\nolimits_{#1}}%
          {\mathop{\kern 0.1em\vrule width 0.3em height 0.697ex depth -0.604ex
                  \kern -0.4em \intop}\nolimits_{#1}}}

\usepackage{a4wide}
\usepackage{mathtools}
\usepackage{amsthm}
\usepackage{amssymb}
\usepackage{hyperref}
\usepackage{color}
\mathtoolsset{showonlyrefs}
\usepackage[obeyFinal]{todonotes}
\usepackage{csquotes}
\usepackage{graphicx}
\usepackage{epsfig,color}
\usepackage{enumitem}

\numberwithin{equation}{section}
\newtheorem{theorem}{Theorem}[section]
\newtheorem{lemma}[theorem]{Lemma}
\newtheorem{corollary}[theorem]{Corollary}
\theoremstyle{definition}

\theoremstyle{remark}
\newtheorem{rem}[theorem]{Remark}

\newcommand{\N}{\mathbb{N}}

\newcommand{\diam}{\mathrm{diam}}

\begin{document}

\title[Restricting open surjections]
{Restricting open surjections}

\author{Jesus A. Jaramillo}
\address{Instituto de Matem\'atica Interdiscliplinar (IMI) and Departamento de An\'alisis Matem\'atico\\
Universidad Complutense de Madrid\\
28040-Madrid\\
Spain}
\email{jaramil@mat.ucm.es}

\author{Enrico Le Donne}
\author{Tapio Rajala}

\address{University of Jyvasklya\\
         Department of Mathematics and Statistics \\
         P.O. Box 35 (MaD) \\
         FI-40014 University of Jyvaskyla \\
         Finland}
\email{ledonne@msri.org}
\email{tapio.m.rajala@jyu.fi}

\thanks{The research of Jaramillo is supported in part by MINECO grant MTM2015-65825-P (Spain). Le Donne acknowledges the support of the Academy of Finland, project no. 288501, and the European Research Council, ERC-StG grant GeoMeG. Rajala acknowledges the support of the Academy of Finland, project  no. 274372.}
\date{\today}

\begin{abstract}
We show that any continuous open surjection from a complete metric space to another metric space can be restricted to a surjection for which the domain has the same density character as the target. This improves a recent result of Aron, Jaramillo and Le Donne.
\end{abstract}

\maketitle

\section{Introduction}

In recent years, there has been an increasing interest in studying {\it surjectivity properties} 
of different classes of maps between Banach spaces or, more generally, metric spaces. For
a metric space $X$ we denote by ${\rm card}(X)$
the cardinality of $X$ and by ${\rm dens}(X)$
the density character, which is defined as the smallest cardinality of a dense subset of $X$.

In the Banach-space setting, it was proved in \cite{Ba} that every separable Banach 
space $Y$ is the range of a $C^1$-smooth surjection $f:X \to Y$ from any infinite-dimensional separable Banach space $X$. In addition, some conditions are given in \cite{Ba} 
under which $f$ can be chosen to be $C^\infty$-smooth. In \cite{Ha}, 
it was shown that if $X=c_0$ and $Y=\ell_2$, then $f$ cannot be $C^2$-smooth.
For the non-separable space $X=c_0(\omega_1)$,  the results of \cite{GHM} show that 
the existence of $C^2$-smooth surjections onto $\ell_2$ depends on additional axioms 
of set theory.

The problem of surjectivity of separable restrictions has been considered
in \cite{AJR}. It is proved there that, for every $X$ belonging to some class of Banach spaces  which includes  
all  $C^\infty$-smooth spaces with density $\geq 2^{\aleph_0}$, as well as all spaces
$\ell_p(\Gamma)$ with  ${\rm card} (\Gamma) \geq 2^{\aleph_0}$, and for every Banach space
$Y$ with dimension $\geq 2$, there exists a $C^\infty$ smooth surjection $f: X \to Y$ whose 
restriction to any separable subspace of $X$ fails to be surjective.  More recently, this result
has been extended in \cite{HJ}, where it is shown to hold for every nonseparable 
super-reflexive space $X$.

In the  metric setting, positive results about the surjectivity of separable restrictions
have been obtained in \cite{ALJ2017}. A map $f \colon X \to Y$ between metric spaces is called
{\it density-surjective} if there is a subset $Z \subset X$ so that ${\rm dens}(Z) = {\rm dens}
(Y)$ and $f|_Z \colon Z \to Y$ is surjective. It is shown in \cite{ALJ2017}  that every 
{\it uniformly open} continuous surjection from a complete metric space to another metric space 
is density-surjective. This result has been refined very recently in \cite{KR2017}, where it is proved that the corresponding surjective restriction of a uniformly open surjection can be also chosen to be  uniformly open.

In this short note we improve on the mentioned result from \cite{ALJ2017} by replacing the uniformly openess assumption by openess. 

\begin{theorem}\label{thm:main}
Let $X$ and $Y$ be metric spaces with $X$ complete.
Let $f \colon X \to Y$ be a continuous open surjection. Then $f$ is density-surjective.
\end{theorem}

We will actually prove a slightly more general version of Theorem \ref{thm:main}. This is stated in Theorem \ref{thm:main:extended}. 

Recall that, if $f: M \to N$ is a $C^1$-smooth map between Banach manifolds, a point $x\in M$ is said to be a {\it regular point} of $f$ if the derivative $d f(x): T_x M \to T_{f(x)} N$ is onto. A point $y\in N$ is said to be a {\it regular value} of $f$ if every $x\in f^{-1}(y)$ is a regular point. Also, note that every {\it paracompact} Banach manifold admits a complete metric (see e.g., 
\cite[Corollary in page 2]{Pal}). Now as a corollary of Theorem \ref{thm:main:extended} and the open mapping theorem for Banach manifolds (see e.g.,  in \cite[Theorem 3.5.2]{AMR} we obtain the following consequence.




\begin{corollary}	
Let $f : M \to N$ be a $C^1$-smooth surjection between paracompact Banach manifolds, and suppose that the set of critical values of $f$ is countable. Then $f$ is density-surjective.
\end{corollary}





\section{Proof of the Theorem}

We will prove the following, slightly more general version of Theorem \ref{thm:main}

\begin{theorem}\label{thm:main:extended}
Let $X$ and $Y$ be metric spaces and $Y'\subset Y$ a subset.
Assume that $X$ is complete. 
Let $f \colon X \to Y$ be continuous such that for $X':= f^{-1}(Y')$  the map $f|_{X'} \colon X' \to Y'$ is open and surjective.
Then there exists a subspace $X_0 \subset X'$ such that $f|_{X_0} \colon X_0 \to Y'$ is surjective with ${\rm dens}(X_0)={\rm dens}(Y')$ and $X_0 $ is relatively closed in $ X'$.
\end{theorem}

First of all, without loss of generality we may assume that the density character of $Y'$ is at least $\aleph_0$.

\begin{lemma}\label{lma:1:re}
 Let $\tilde X$ and $\tilde Y$ be metric spaces, $f \colon \tilde X \to \tilde Y$ a continuous open surjection and  $r > 0$.
Then there exists an open cover $\{\tilde Y_k\}_{k \in I}$ of $\tilde Y$, with  ${\rm card}(I) \le {\rm dens}(\tilde Y)$ such that for every $k  \in I$ we have $\diam(\tilde Y_k) \le r$  and there exists a point 
$\tilde x_k \in f^{-1}(\tilde Y_k)$ such that 
\begin{equation}\label{eq:cover}
\tilde Y_k \subset f(B(\tilde x_k,r)).
\end{equation}
Consequently, we also have that for each $k\in I$ we have that the map 
\begin{equation}\label{map}f:f^{-1}(\tilde Y_k) \cap B(\tilde x_k,r) 
\to \tilde Y_k
\end{equation}
is a continuous and open surjection.
\end{lemma}
\begin{proof}
For each $i \in \N$ define
\[
\tilde X_i := \{\tilde x \in \tilde X \,:\, B(f(\tilde x),2^{-i}r) \subset f(B(\tilde x,r))\}.
\]
Since $f$ is open, we have $\tilde X = \bigcup_{i \in \N} \tilde X_i$. It suffices to cover the set
\[
\tilde Y^i := f(\tilde X_i)
\]
for each $i \in \N$. Let $D^i = \{\tilde y_{k}^i\}_{k \in I_i}$ be dense set in $\tilde Y^i$ with 
${\rm card}(I_i) \le {\rm dens}(\tilde Y)$. For each $k \in I_i$ select a point $\tilde x_k^i \in f^{-1}(\{\tilde y_{k}^i\}) \cap  \tilde X_i$ and define 
\[
\tilde Y_k^i = B(\tilde y_{k}^i,2^{-i}r).
\]
By construction, $\tilde Y_k^i$ are open, $\tilde Y^i = \bigcup_{k \in I_i}\tilde Y_k^i$, $\diam(\tilde Y_{k}^i) \le 2^{1-i}r \le r$, $\tilde x_k^i \in f^{-1}(\tilde Y_{k}^i)$ and $\tilde Y_{k}^i \subset f(B(\tilde x_k^i,r))$ as was required.
From \eqref{eq:cover}, the map in \eqref{map} is surjective. Moreover, it is continuous and open since it is the restriction to an open set of a continuous and open map. 
\end{proof}

\begin{rem}
 Notice that in the proof of Lemma \ref{lma:1:re} one cannot directly use the openness of $f$ to open neighborhoods of preimages of a dense set of points in $\tilde Y$ since the images of such neighborhoods need not cover the space $\tilde Y$. This is why the space $\tilde X$ is first covered by sets $\tilde X_i$ where the map is uniformly open on the given scale $r$.
\end{rem}

\begin{proof}[{\bf Proof of Theorem \ref{thm:main:extended}}]
 We prove the claim by repeatedly using Lemma \ref{lma:1:re}. First, we let $r = 1$, $\tilde X = X'$ and $\tilde Y = Y'$. Lemma \ref{lma:1:re} gives us an open cover 
 $\{Y_k^0\}_{k \in I_0}$ of $Y'$ and points $\{x_k^0\}_{k \in I_0}$.
 
 Then we continue inductively. Suppose we have defined for fixed $i$ an open cover $\{Y_k^i\}_{k \in I_i}$ of $Y'$ and corresponding points $\{x_k^{i}\}_{k \in I_{i}}$. We continue to cover each $Y_k^i$ with the help of Lemma \ref{lma:1:re} by taking $r = 2^{-i}$, $\tilde Y = Y_k^i$
and $\tilde X = f^{-1}(Y_k^i) \cap B(x_k^i,2^{-i})$.
Lemma \ref{lma:1:re} now gives us a cover $\{Y_{k,j}^i\}_{j \in J_{k,i}}$ of $Y_k^i$ by sets $Y_{k,j}^i \subset Y_k^i$. We collect all the covers of $Y_k^i$, for $k \in I_i$, and write the collection as
\[
\{Y_k^{i+1}\}_{k \in I_{i+1}} = \bigcup_{k \in I_i}\bigcup_{j \in J_{k,i}}\{Y_{k,j}^i\},
\]
and similarly we collect all the corresponding points $\{x_k^{i+1}\}_{k \in I_{i+1}}$ for which \eqref{eq:cover} holds.

Now we define $X_0$ as the closure in $X'$ of
\[
 D = \bigcup_{i=0}^\infty\bigcup_{k \in I_i} \{x_k^i\}.
\]
Since for each $i \in \N$ we have ${\rm card}(I_i) \le {\rm dens}(Y')$, we have ${\rm card}(D) \le {\rm dens}(Y')$. Thus ${\rm dens}(X_0) \le {\rm dens}(Y')$. 

We still need to verify that $f(X_0) = Y'$. 
For this, take $y \in Y'$.
For each $i \in \N$ we select an index $k_i \in I_i$ inductively.
Since $\{Y_k^0\}_{k \in I_0}$ covers $Y'$, there exists $k_0 \in I_0$ such that $y \in Y_{k_0}^0$. 
Suppose now that $y \in Y_{k_i}^i$ for $i \ge 0$. Then by construction of the collection $\{Y_k^{i+1}\}_{k \in I_{i+1}}$, there exists an index $k_{i+1} \in I_{i+1}$ such that $y \in Y_{k_i}^i$ and  $Y_{k_{i+1}}^{i+1} \subset Y_{k_i}^i$.
Thus, by writing $y_i := f(x_{k_i}^i) \in Y_{k_i}^i$ we obtain a sequence $(y_i)_{i \in \N}$ such that $d(y_i,y)\le 2^{-i}$. Moreover, by construction
$x_{k_{i+1}}^{i+1} \in B(x_{k_{i}}^{i},2^{-i})$ for all $i \in \N$. Therefore, $(x_{k_i}^i)_{i\in\N}$ is a Cauchy sequence.
Since $X$ is complete, there exists $ x \in X$ such that $x_{k_i}^i \to x $ as $i \to \infty$. By continuity of $f$ we have $f(x) = y$ and thus $x \in X'$.
\end{proof}

\end{document}